\documentclass[12pt, reqno]{amsart}
\usepackage{amsmath, amsthm, amscd, amsfonts, amssymb, graphicx, color}
\usepackage[bookmarksnumbered, colorlinks, plainpages]{hyperref}

\newtheorem{theorem}{Theorem}[section]

\newtheorem{proposition}[theorem]{Proposition}
\newtheorem{corollary}[theorem]{Corollary}
\theoremstyle{definition}

\theoremstyle{remark}

\numberwithin{equation}{section}

\begin{document}
\setcounter{page}{1}


\title[on THE NON-COMMUTATIVE NEWTON BINOMIAL FORMULA
] {on THE NON-COMMUTATIVE NEWTON BINOMIAL FORMULA
}
\author[A. Hosseini]{A. Hosseini$^{\ast}$}
\address{ A. Hosseini, Department of Mathematics, Kashmar Higher Education Institute, Kashmar, Iran}
\email{\textcolor[rgb]{0.00,0.00,0.84}{hussi.kashm@gmail.com}}

\author[M. Mohammadzadeh Karizaki]{M. Mohammadzadeh Karizaki}
\address{ M. Mohammadzadeh Karizaki, University of Torbat Heydarieh, Torbat heydariyeh,
Iran}
\email{\textcolor[rgb]{0.00,0.00,0.84}{ m.mohammadzadeh@torbath.ac.ir}}


\subjclass[2010]{47B47, 47B48, 11B65}

\keywords{Newton's binomial formula; non-commutative Newton's binomial formula; derivation; generalized derivation; Banach algebra.}

\date{Received: xxxxxx; Revised: yyyyyy; Accepted: zzzzzz.
\newline \indent $^{*}$ Corresponding author}

\begin{abstract}
In this article, using generalized derivations, we obtain a simple idea to prove the non-commutative
Newton binomial formula in unital algebras and then, we extend that formula
to non-unital algebras. Additionally, we establish the non-commutative Newton binomial formula with a negative power.
\end{abstract} \maketitle

\section{Introduction and preliminaries}
Throughout this article, let $\mathcal{A}$ be an associative complex algebra. If $\mathcal{A}$ is unital, then \textbf{1} stands for the identity  element. In this paper, we present an interesting relationship between generalized derivations and the non-commutative Newton's binomial formula. So, the concept of a generalized derivation plays a fundamental role in this regard. Recall that a linear mapping $d:\mathcal{A} \rightarrow \mathcal{A}$ is called a derivation if $d(ab) = d(a)b + ad(b)$ for all $a, b \in \mathcal{A}$. For example, let $b$ be an arbitrary fixed element of $\mathcal{A}$. A linear mapping $d_b:\mathcal{A} \rightarrow \mathcal{A}$ defined by $d_b(a) = ba - ab$ for any $a \in \mathcal{A}$, is a derivation which is called an inner derivation. Now we state the concept of a generalized derivation which was introduced by Bre$\check{s}$ar \cite{B4}. A linear mapping $\delta: \mathcal{A} \rightarrow \mathcal{A}$ is called a generalized derivation if there exists a derivation $d:\mathcal{A} \rightarrow \mathcal{A}$ such that $\delta$ satisfies $\delta(ab) = \delta(a) b + a d(b)$ for all $a, b \in \mathcal{A}$. Let $b_1, b_2$ be two arbitrary elements of $\mathcal{A}$. A linear mapping $\delta_{b_1, b_2}:\mathcal{A} \rightarrow \mathcal{A}$ defined by $\delta_{b_1, b_2}(a) = b_1 a - a b_2$ ($a \in \mathcal{A}$) is a generalized derivation associated with both the inner derivations $d_{b_1}$ and $d_{b_2}$. Indeed, we have $\delta_{b_1, b_2}(ab) = \delta_{b_1, b_2}(a)b + a d_{b_2}(b) = d_{b_1}(a) b + a \delta_{b_1, b_2}(b)$ for all $a, b \in \mathcal{A}$. We say that $\delta_{b_1, b_2}$ is an inner generalized derivation. As seen, a generalized derivation $\delta$ can satisfy $\delta(ab) = a \delta(b) + d(a) b$ for all $a, b \in \mathcal{A}$.
In order to complete the introduction, we recall the Newton's binomial formula. let $a$ and $b$ be two elements of $\mathcal{A}$ which commute with each other, i.e. $ab = ba$. Newton's binomial formula is as follows:\\
$$(a + b)^n = \sum_{k = 0}^{n}\Big(_{k}^{n}\Big)a^{n - k}b^{k},$$ where $\Big(_{k}^{n}\Big) = \frac{n!}{(n - k)! k!}$. Since the binomial coefficients are integers, this identity makes sense in an arbitrary ring with $\textbf{1}$. Newton's binomial formula occurs naturally in varied contexts, including combinatorics, number theory, and mathematical physics. That is why we consider it as a very important formula in Mathematics. The non-commutative binomial formula may be played a fundamental role in quantum group theory (see \cite{M1, M2} for more details in this issue). So far, several generalizations have been given for the Newton's binomial formula which we will state them below.

An important generalization is obtained by considering polynomials in two non-commuting variables $x$ and $y$ satisfying $xy = q yx$ where $q$ is a scalar. In that case one has
\begin{align}
(x + y)^n = \sum_{k = 0}^{n}\frac{(q ; q)_n}{(q ; q)_k (q ; q)_{n  k}}y^{k} x^{n - k},
\end{align}
where we use the standard notation $$(a ; q)_k = \prod_{i = 0}^{k - 1}( 1 - a q^i).$$ Since the $q$-binomial coefficients occurring in (1.1) are polynomials in $q$ (the Gaussian polynomials), this identity makes sense for $x, y$, and $q$ in an arbitrary ring, assuming also
that $q$ commutes with $x$ and $y$. The $q$-binomial formula (1.1) was first given explicitly by
Sch$\ddot{u}$tzenberger \cite{Sc}. In 1999, Benaoum \cite{B1} proved the general formula below:
$$(x + y)^n = \sum_{k = 0}^{n}\frac{(q ; q)_n}{(q ; q)_k (q ; q)_{n  k}}\prod_{j = 0}^{k - 1}\Big( 1 + [j]_{q} h \Big)y^{k} x^{n - k},$$ where $[j]_{q} = 1 + q + ... + q^{j - 1}$ and $x, y$ satisfy $xy = q yx + hy^2$. For more informative details in this regard see \cite{B1, B2, C} and references therein. Recently, Walter Wyss \cite{W1} has obtained an interesting non-commutative Newton's binomial formula on a unital algebra $\mathcal{A}$ as follows:\\
\begin{align*}
(a + b)^n = \sum_{k = 0}^{n}\Big(_{k}^{n}\Big)\Bigg[\Big(a + d_b\Big)^{k}(\textbf{1})\Bigg]b^{n - k},
\end{align*}
where $d_b: \mathcal{A} \rightarrow \mathcal{A}$ is an inner derivation. In this paper, we present a procedure by which one can easily obtain the Wyss formula. Indeed, our formula is the extension of the Wyss formula. In short, our discussion is as follows. Let $\mathcal{A}$ be a unital algebra and let $a, b$ be two arbitrary elements of $\mathcal{A}$. Then for any non-negative integer $n$, it holds that
\begin{align*}
a^{n} = \sum_{k = 0}^{n}\Big(_{k}^{n}\Big)\delta_{a, b}^{n - k}(\textbf{1})b^{k},
\end{align*}
where $\delta_{a, b}$ is an inner generalized derivation. Replacing $a + b$ instead of $a$ and $c$ instead of $b$ in the previous formula, we get that
\begin{align*}
(a + b)^n & = \sum_{k = 0}^{n}\Big(_{k}^{n}\Big)\delta_{a + b, c}^{n - k}(\textbf{1})c^{k},
\end{align*}
for any $a, b , c \in \mathcal{A}$. Moreover, we achieve a non-commutative Newton's binomial formula in the non-unital algebras. Eventually, we obtain a non-commutative binomial formula with a negative power. Indeed, under certain conditions, we show that if $\mathcal{A}$ is a unital Banach algebra, then for any $a, b \in \mathcal{A}$ and $n \in \mathbb{N}$ we have
\begin{align*}
(a + b)^{-n} & = \sum_{k = 0}^{\infty}\Big(_{k}^{-n}\Big)\lambda^{-n - k}(a + b_1)^{k} \\ & = \sum_{k = 0}^{\infty}\sum_{j = 0}^{k}\Big(_{k}^{-n}\Big)\Big(_{j}^{k}\Big)\lambda^{-n - k}\delta_{a + b_1, b_1}^{k - j}(\textbf{1})b_1^{j},
\end{align*}
where $b_1 = b - \lambda \textbf{1}$, $\lambda$ is a suitable complex number and $$\Big(_{k}^{-n}\Big) = \frac{(-n)(-n - 1) ... (-n - k + 1)}{k!} = \frac{(-1)^{k}n(n + 1) ... (n + k - 1)}{k!}.$$
\section{results and proofs}
Throughout this section, without further mention, if an algebra is unital, then $\textbf{1}$ stands for the identity element. Recall that if $b_1, b_2$ are two arbitrary fixed elements of an algebra $\mathcal{A}$, then a linear mapping $\delta_{b_1, b_2}:\mathcal{A} \rightarrow \mathcal{A}$ defined by $\delta_{b_1, b_2}(a) = b_1 a - a b_2$ ($a \in \mathcal{A}$) is a generalized derivation associated with both the inner derivations $d_{b_1}$ and $d_{b_2}$ (see Introduction). The set of all invertible elements of an algebra $\mathcal{A}$ is denoted by $Inv(\mathcal{A})$. Let $b \in \mathcal{A}$ and $c \in Inv(\mathcal{A})$. We introduce a linear mapping $\alpha_{b, c}:\mathcal{A} \rightarrow \mathcal{A}$ by $\alpha_{b, c}(a) = b a c^{-1}$. A straightforward verification shows that $\alpha_{b, c}(a_1 a_2) = \alpha_{b, c}(a_1)\alpha_{c, c}(a_2)$ for all $a_1, a_2 \in \mathcal{A}$. In the following, we first prove an auxiliary result which will be used extensively to conjecture the non-commutative Newton's binomial formula in any unital algebra. The proposition below has been proved in \cite{H}, but in order to make this paper self contained, we state it with its proof here. Recall that Pascal's Formula (or Pascal's Rule) is the equation $$\Big(^{n}_{k}\Big) + \Big( ^{n}_{k - 1} \Big) = \Big(_{k}^{n + 1}\Big)$$ for $n, k \in \mathbb{Z}_{+}$ and $k < n$. Let $\mathcal{A}$ be a unital Banach algebra. It is a well-known fact in the theory of Banach algebras that $e^{a} = \sum_{n = 0}^{\infty}\frac{a^n}{n !}$ for any $a \in \mathcal{A}$.
\\
\\
\begin{proposition}\label{1} Let $\mathcal{A}$ be a unital Banach algebra, and let $b_1, b_2$ be two arbitrary elements of $\mathcal{A}$. Then, $e^{\delta_{b_1, b_2}} = \alpha_{e^{b_1}, e^{b_2}}$.
\end{proposition}

\begin{proof} First, by induction we show that
\begin{align}
\delta_{b_1, b_2}^{n}(a) = \sum_{k = 0}^{n}(-1)^{k}\Big(_{k}^{n}\Big)b_1^{n - k} a b_2^{k},
\end{align}
for each non-negative integer $n$ and any $a \in \mathcal{A}$. Obviously, (2.1) is true for $n = 1$. Suppose that the statement (2.1) holds for $n$ (Induction Hypothesis). Our task is to prove that (2.1) is also true for $n + 1$. Using induction hypothesis and Pascal's Formula, we have the following expressions:
\begin{align*}
\delta_{b_1, b_2}^{n + 1}(a) & = \delta_{b_1, b_2}\Big(\delta_{b_1, b_2}^{n}(a)\Big) \\ & = b_1 \delta_{b_1, b_2}^{n}(a) - \delta_{b_1, b_2}^{n}(a) b_2 \\ & = b_1 \sum_{k = 0}^{n}(-1)^{k}\Big(_{k}^{n}\Big)b_1^{n - k} a b_2^{k} - \sum_{k = 0}^{n}(-1)^{k}\Big(_{k}^{n}\Big)b_1^{n - k} a b_2^{k} b_2 \\ & = \sum_{k = 0}^{n}(-1)^{k}\Big(_{k}^{n}\Big)b_1^{n + 1 - k} a b_2^{k} - \sum_{k = 0}^{n}(-1)^{k}\Big(_{k}^{n}\Big)b_1^{n - k} a b_2^{k + 1} \\ & = \sum_{k = 0}^{n}(-1)^{k}\Big(_{k}^{n}\Big)b_1^{n + 1 - k} a b_2^{k} + \sum_{k = 1}^{n + 1}(-1)^{k}\Big(_{k - 1}^{n}\Big)b_1^{n - k + 1} a b_2^{k} \\ & = \sum_{k = 1}^{n}(-1)^{k} \Bigg[\Big(_{k}^{n}\Big) + \Big(_{k - 1}^{n}\Big)\Bigg]b_1^{n + 1 - k} a b_2^{k} + b_1^{n + 1}a + (-1)^{n + 1}a b_2^{n + 1} \\ & = \sum_{k = 1}^{n}(-1)^{k}\Big(_{k}^{n + 1}\Big)b_1^{n + 1 - k} a b_2^{k} +  b_1^{n + 1}a + (-1)^{n + 1}a b_2^{n + 1} \\ & = \sum_{k = 0}^{n}(-1)^{k}\Big(_{k}^{n + 1}\Big)b_1^{n + 1 - k} a b_2^{k}
\end{align*}
It means that equation (2.1) is true for all $a, b_1, b_2 \in \mathcal{A}$ and $n \in \mathbb{N} \cup \{0\}$. Notice that $\delta_{b_1, b_2} \in B(\mathcal{A})$, the set of all bounded linear mappings from $\mathcal{A}$ into itself, and so, $e^{\delta_{b_1, b_2}} = \sum_{n = 0}^{\infty}\frac{\delta_{b_1, b_2}^{n}}{n !}$. Let $a$ be an arbitrary element of $\mathcal{A}$. Therefore, we have
\begin{align*}
e^{\delta_{b_1, b_2}}(a) & = \sum_{n = 0}^{\infty}\frac{\delta_{b_1, b_2}^{n}(a)}{n !} \\ & = \sum_{n = 0}^{\infty}\frac{1}{n!}\sum_{k = 0}^{n}(-1)^{k}\Big(_{k}^{n}\Big)b_1^{n - k} a b_2^{k} \\ & = \sum_{n = 0}^{\infty}\sum_{k = 0}^{n}\frac{(-1)^k n !}{n !(n - k)! k!}b_1^{n - k} a b_2^{k} \\ & = \sum_{n = 0}^{\infty}\sum_{k = 0}^{n}\Big(\frac{b_1^{n - k}}{(n - k)!}\Big) a \Big(\frac{(-1)^k b_2^k}{k!}\Big) \\ & = \Big(\sum_{n = 0}^{\infty}\frac{b_1^n}{n!}\Big) a \Big(\sum_{n = 0}^{\infty}\frac{(-b_2)^n}{n!}\Big) \\ & = e^{b_1} a e^{-b_2} = e^{b_1} a (e^{b_2})^{-1} \\ & = \alpha_{e^{b_1}, e^{b_2}}(a),
\end{align*}
which means that $e^{\delta_{b_1, b_2}} = \alpha_{e^{b_1}, e^{b_2}}$. Thereby, our goal is achieved.
\end{proof}
Let $a$ and $b$ be two arbitrary elements of $\mathcal{A}$. Using Proposition \ref{1}, we have the following statements:
\begin{align*}
\sum_{n = 0}^{\infty}\frac{a^n}{n !} & = e^{a} = e^{a}e^{-b}e^{b} = e^{a} \textbf{1} e^{-b}e^{b} \\ & = \alpha_{e^{a}, e^{b}}(\textbf{1}) e^{b} = e^{\delta_{a, b}}(\textbf{1}) e^{b} \\ & = \sum _{n = 0}^{\infty}\frac{\delta_{a, b}^{n}(\textbf{1})}{n!}\sum_{n = 0}^{\infty}\frac{b^n}{n !} \\ & = \sum_{n = 0}^{\infty}\sum_{k = 0}^{n}\frac{\delta_{a, b}^{n - k}(\textbf{1})}{(n - k)!}\frac{b^k}{k!} \\ & = \sum_{n = 0}^{\infty}\frac{1}{n!}\sum_{k = 0}^{n}\frac{n!\delta_{a, b}^{n - k}(\textbf{1})b^{k}}{(n - k)!k!} \\ & = \sum_{n = 0}^{\infty}\frac{1}{n!}\sum_{k = 0}^{n}\Big(_{k}^{n}\Big)\delta_{a, b}^{n - k}(\textbf{1})b^{k},
\end{align*}
which means that
\begin{align}
\sum_{n = 0}^{\infty}\frac{a^n}{n !} = \sum_{n = 0}^{\infty}\frac{1}{n!}\sum_{k = 0}^{n}\Big(_{k}^{n}\Big)\delta_{a, b}^{n - k}(\textbf{1})b^{k},
\end{align}
for any $a, b \in \mathcal{A}$. Using equation (2.2), we guess the following formula:

$$a^{n} = \sum_{k = 0}^{n}\Big(_{k}^{n}\Big)\delta_{a, b}^{n - k}(\textbf{1})b^{k},$$ for any $a, b \in \mathcal{A}$. In the next theorem, we prove the above formula in any unital algebra.
\begin{theorem} \label{2} Let $\mathcal{A}$ be a unital algebra. Then for any $a, b \in \mathcal{A}$ and any non-negative integer $n$, it holds that
\begin{align}
a^{n} = \sum_{k = 0}^{n}\Big(_{k}^{n}\Big)\delta_{a, b}^{n - k}(\textbf{1})b^{k}.
\end{align}
\end{theorem}

\begin{proof} Obviously, the above-mentioned formula holds true for $n = 0, 1$. We now proceed by induction. Suppose that (2.3) is true for an arbitrary positive integer $n$. We are going to show that $a^{n + 1} = \sum_{k = 0}^{n + 1}\Big(_{k}^{n + 1}\Big)\delta_{a, b}^{n + 1 - k}(\textbf{1})b^{k}$. We have the following expressions:
\begin{align*}
\sum_{k = 0}^{n + 1}\Big(_{k}^{n + 1}\Big)\delta_{a, b}^{n + 1 - k}(\textbf{1})b^{k} & = b^{n + 1} + \sum_{k = 1}^{n}\Big(_{k}^{n + 1}\Big)\delta_{a, b}^{n + 1 - k}(\textbf{1})b^{k} + \delta_{a, b}^{n + 1}(\textbf{1}) \\ & = b^{n + 1} + \sum_{k = 1}^{n}\Bigg[\Big(_{k}^{n}\Big) + \Big(_{k - 1}^{n}\Big)\Bigg]\delta_{a, b}^{n + 1 - k}(\textbf{1})b^{k} + \delta_{a, b}^{n + 1}(\textbf{1}) \\ & = b^{n + 1} + \sum_{k = 1}^{n}\Big(_{k}^{n}\Big) \delta_{a, b}\Bigg(\delta_{a, b}^{n - k}(\textbf{1})\Bigg)b^{k} + \\ & \sum_{k = 1}^{n}\Big(_{k - 1}^{n}\Big)\delta_{a, b}^{n + 1 - k}(\textbf{1})b^{k} + \delta_{a, b}^{n + 1}(\textbf{1})\\ & = b^{n + 1} + \sum_{k = 1}^{n}\Big(_{k}^{n}\Big) \Bigg[\delta_{a, b}\Bigg(\delta_{a, b}^{n - k}(\textbf{1})b^{k}\Bigg) - \delta_{a, b}^{n - k}(\textbf{1})d_{b}(b^{k})\Bigg] + \\ & \sum_{k = 1}^{n}\Big(_{k - 1}^{n}\Big)\delta_{a, b}^{n + 1 - k}(\textbf{1})b^{k} + \delta_{a, b}^{n + 1}(\textbf{1}) \\ & = b^{n + 1} + \sum_{k = 1}^{n}\Big(_{k}^{n}\Big) \delta_{a, b}\Bigg(\delta_{a, b}^{n - k}(\textbf{1})b^{k}\Bigg) - \sum_{k = 1}^{n}\Big(_{k}^{n}\Big)\delta_{a, b}^{n - k}(\textbf{1})d_{b}(b^{k}) \\ & + \sum_{k = 1}^{n}\Big(_{k - 1}^{n}\Big)\delta_{a, b}^{n + 1 - k}(\textbf{1})b^{k} + \delta_{a, b}^{n + 1}(\textbf{1}) \\ & = b^{n + 1} + \delta_{a, b}\Bigg(\sum_{k = 1}^{n}\Big(_{k}^{n}\Big) \delta_{a, b}^{n - k}(\textbf{1})b^{k}\Bigg) - \sum_{k = 1}^{n}\Big(_{k}^{n}\Big)\delta_{a, b}^{n - k}(\textbf{1})d_{b}(b^{k}) \\ & + \sum_{k = 1}^{n}\Big(_{k - 1}^{n}\Big)\delta_{a, b}^{n + 1 - k}(\textbf{1})b^{k} + \delta_{a, b}^{n + 1}(\textbf{1})\\
& = b^{n + 1} + \delta_{a, b}\Bigg[\sum_{k = 0}^{n}\Big(_{k}^{n}\Big) \delta_{a, b}^{n - k}(\textbf{1})b^{k} - \delta_{a, b}^{n}(\textbf{1})\Bigg] - \sum_{k = 1}^{n}\Big(_{k}^{n}\Big)\delta_{a, b}^{n - k}(\textbf{1})d_{b}(b^{k}) \\ & + \sum_{k = 1}^{n}\Big(_{k - 1}^{n}\Big)\delta_{a, b}^{n + 1 - k}(\textbf{1})b^{k} + \delta_{a, b}^{n + 1}(\textbf{1}) \\
& = b^{n + 1} + \delta_{a, b}\Bigg[a^{n} - \delta_{a, b}^{n}(\textbf{1})\Bigg] - \sum_{k = 1}^{n}\Big(_{k}^{n}\Big)\delta_{a, b}^{n - k}(\textbf{1})(b b^{k} - b^{k}b) \\ & + \sum_{k = 1}^{n}\Big(_{k - 1}^{n}\Big)\delta_{a, b}^{n + 1 - k}(\textbf{1})b^{k} + \delta_{a, b}^{n + 1}(\textbf{1}) \\ & = b^{n + 1} + \delta_{a, b}(a^n) - \delta_{a, b}^{n + 1}(\textbf{1}) + \sum_{k = 1}^{n}\Big(_{k - 1}^{n}\Big)\delta_{a, b}^{n + 1 - k}(\textbf{1})b^{k} + \delta_{a, b}^{n + 1}(\textbf{1}) \\
 \end{align*}
 \begin{align*}
 & = b^{n + 1} + \Big(a a^n - a^{n}b\Big) + \sum_{k = 1}^{n}\Big(_{k - 1}^{n}\Big)\delta_{a, b}^{n + 1 - k}(\textbf{1})b^{k} \\ & = b^{n + 1} + a^{n + 1} - a^{n}b + \sum_{k = 0}^{n - 1}\Big(_{k}^{n}\Big)\delta_{a, b}^{n - k}(\textbf{1})b^{k + 1} \\ & = b^{n + 1} + a^{n + 1} - a^{n}b + \Bigg[\sum_{k = 0}^{n}\Big(_{k}^{n}\Big)\delta_{a, b}^{n - k}(\textbf{1})b^{k} - b^{n}\Bigg]b \\ & = b^{n + 1} + a^{n + 1} - a^{n}b + \Bigg[a^{n} - b^{n}\Bigg]b \\ & = b^{n + 1} + a^{n + 1} - a^{n}b + a^{n}b - b^{n + 1} \\ & = a^{n + 1}.
\end{align*}
It means that $a^{n + 1} = \sum_{k = 0}^{n + 1}\Big(_{k}^{n + 1}\Big)\delta_{a, b}^{n + 1 - k}(\textbf{1})b^{k}$. Thereby, our assertion is completely proved.
\end{proof}

In the following, we present some consequences of the above theorem.

\begin{corollary} Let $\mathcal{A}$ be a unital algebra, and let $a, b$ be two arbitrary elements of $\mathcal{A}$. Then \\
i) $a^{n} - b^{n} = \sum_{k = 0}^{n - 1}\sum_{j = 0}^{n - k}\Big(_{k}^{n}\Big)\Big(_{j}^{n - k}\Big)(-1)^{j}a^{n - k - j}b^{j + k}$ for any $n \geq 1$,\\
ii) (The non-commutative Newton's binomial formula)
\begin{align*}
(a + b)^n & = \sum_{k = 0}^{n}\Big(_{k}^{n}\Big)\delta_{a + b, c}^{n - k}(\textbf{1})c^{k} \\ & = \sum_{k = 0}^{n}\sum_{j = 0}^{n - k}\Big(_{k}^{n}\Big)\Big(_{j}^{n - k}\Big)(-1)^{j}(a + b)^{n - k - j}c^{j + k},
\end{align*}
where $c$ is an arbitrary element of $\mathcal{A}$ and $n \in \mathbb{N} \cup \{0\}$.
\end{corollary}
\begin{proof}(i) By Theorem \ref{2} we know that
\begin{align*}
a^{n} = \sum_{k = 0}^{n}\Big(_{k}^{n}\Big)\delta_{a, b}^{n - k}(\textbf{1})b^{k} \ for \ all \ a, b \in \mathcal{A}.
\end{align*}
It follow from Proposition \ref{1} that
\begin{align*}
\delta_{a, b}^{n}(c) = \sum_{k = 0}^{n}(-1)^{k}\Big(_{k}^{n}\Big)a^{n - k} c b^{k},
\end{align*}
for each non-negative integer $n$ and $a, b, c \in \mathcal{A}$. By considering this formula, we obtain that
\begin{align*}
a^{n} & = \sum_{k = 0}^{n}\Big(_{k}^{n}\Big)\delta_{a, b}^{n - k}(\textbf{1})b^{k} \\ & = \sum_{k = 0}^{n}\sum_{j = 0}^{n - k}\Big(_{k}^{n}\Big)\Big(_{j}^{n - k}\Big)(-1)^{j}a^{n - k - j}b^{j + k} \\ & = \sum_{k = 0}^{n - 1}\sum_{j = 0}^{n - k}\Big(_{k}^{n}\Big)\Big(_{j}^{n - k}\Big)(-1)^{j}a^{n - k - j}b^{j + k} + b^{n}.
\end{align*}
Therefore, we have
\begin{align*}
a^{n} - b^{n} = \sum_{k = 0}^{n - 1}\sum_{j = 0}^{n - k}\Big(_{k}^{n}\Big)\Big(_{j}^{n - k}\Big)(-1)^{j}a^{n - k - j}b^{j + k},
\end{align*}
for any $a, b \in \mathcal{A}$ and $n \in \mathbb{N}$. \\

(ii) Replacing $a + b$ instead of $a$ and $c$ instead of $b$ in (2.3) and then using Proposition \ref{1}, we get that
\begin{align*}
(a + b)^n & = \sum_{k = 0}^{n}\Big(_{k}^{n}\Big)\delta_{a + b, c}^{n - k}(\textbf{1})c^{k} \\ & = \sum_{k = 0}^{n}\sum_{j = 0}^{n - k}\Big(_{k}^{n}\Big)\Big(_{j}^{n - k}\Big)(-1)^{j}(a + b)^{n - k - j}c^{j + k}.
\end{align*}
In particular, for $c = a, b$ we have
\begin{align*}
(a + b)^n & = \sum_{k = 0}^{n}\Big(_{k}^{n}\Big)\delta_{a + b, b}^{n - k}(\textbf{1})b^{k} \\ & = \sum_{k = 0}^{n}\sum_{j = 0}^{n - k}\Big(_{k}^{n}\Big)\Big(_{j}^{n - k}\Big)(-1)^{j}(a + b)^{n - k - j}b^{j + k} \\ & = \sum_{k = 0}^{n}\sum_{j = 0}^{n - k}\Big(_{k}^{n}\Big)\Big(_{j}^{n - k}\Big)(-1)^{j}(a + b)^{n - k - j}a^{j + k}.
\end{align*}
\end{proof}
 We denote the commutator $ab - ba$ by $[a, b]$ for all $a, b \in \mathcal{A}$. As applications of Theorem \ref{2}, we have
\begin{align*}
& [a, b]^{n} = \sum_{k = 0}^{n}\Big(_{k}^{n}\Big)\delta_{[a,b], c}^{n - k}(\textbf{1})c^{k} = \sum_{k = 0}^{n}\sum_{j = 0}^{n - k}\Big(_{k}^{n}\Big)\Big(_{j}^{n - k}\Big)(-1)^{j}[a, b]^{n - k - j}c^{j + k}, \\ & (ab)^n = \sum_{k = 0}^{n}\sum_{j = 0}^{n - k}\Big(_{k}^{n}\Big)\Big(_{j}^{n - k}\Big)(-1)^{j}(ab)^{n - k - j}c^{j + k},
\end{align*}
for any $a, b, c \in \mathcal{A}$ and $n \in \mathbb{N} \cup \{0\}$.
\\
\\
In the following, we want to obtain a non-commutative binomial formula in a non-unital algebra. An algebra $\mathcal{A}$ can always be embedded into an algebra with identity as follows. Let $\mathfrak{A}$ denote the set of all pairs $(a, \lambda)$, $a \in \mathcal{A}$, $\lambda \in \mathbb{C}$, that is, $\mathfrak{A} = \mathcal{A} \bigoplus \mathbb{C}$. Then $\mathfrak{A}$ becomes an algebra if the linear space operations and
multiplication are defined by
$(a, \lambda) + (b, \mu) = (a + b, \lambda + \mu)$, $\mu(a, \lambda) = (\mu a, \mu \lambda)$
and $(a, \lambda)(b, \mu) = (ab + \lambda b + \mu a, \lambda \mu)$
for $a, b \in \mathcal{A}$ and $\lambda, \mu \in \mathbb{C}$. A simple calculation shows that the element
$\textbf{1} = (0, 1) \in \mathfrak{A}$ is an identity for $\mathfrak{A}$. Moreover, the mapping $a \rightarrow (a, 0)$ is an
algebra isomorphism of $\mathcal{A}$ onto an ideal of codimension one in $\mathfrak{A}$. Obviously,
$\mathfrak{A}$ is commutative if and only if $\mathcal{A}$ is commutative.

Let $\mathcal{A}$ be a non-unital algebra. We introduce a mapping $\Delta_{(a, \alpha), (b, \beta)}: \mathfrak{A} \rightarrow \mathfrak{A}$ by $\Delta_{(a, \alpha), (b, \beta)}(c, \gamma) = (a, \alpha)(c, \gamma) - (c, \gamma)(b, \beta)$ for all $(a, \alpha), (b, \beta), (c, \gamma) \in \mathfrak{A}$. It is evident that $\Delta_{(a, \alpha), (b, \beta)}$ is an inner generalized derivation on the unital algebra $\mathfrak{A}$. Since $\mathfrak{A}$ is unital, it follows from Theorem \ref{2} that
\begin{align*}
(a, \alpha)^{n} = \sum_{k = 0}^{n}\Big(_{k}^{n}\Big)\Delta_{(a, \alpha), (b, \beta)}^{n - k}(0, 1)(b, \beta)^{k},
\end{align*}
and consequently,
\begin{align}
a^{n} = \sum_{k = 0}^{n}\Big(_{k}^{n}\Big)\Delta_{(a, 0), (b, \beta)}^{n - k}(0, 1)(b, \beta)^{k},
\end{align}
for all $a, b \in \mathcal{A}$ and $\alpha, \beta \in \mathbb{C}$. We can thus deduce that
\begin{align}
(a + b)^{n} = \sum_{k = 0}^{n}\Big(_{k}^{n}\Big)\Delta_{(a + b, 0), (c, \gamma)}^{n - k}(0, 1)(c, \gamma)^{k},
\end{align}
for all $a, b, c \in \mathcal{A}$ and $\gamma \in \mathbb{C}$. The above discussion shows that we can also achieve our formulas in non-unital algebras.
\\
\\
In the following, we are going to establish a non-commutative Newton's binomial formula with a negative power. Let $\mathcal{A}$ be a unital algebra, and let $a \in \mathcal{A}$. The spectrum of $a$ is $ \mathfrak{S}(a)= \{\lambda \in \mathbb{C} \ : \ \lambda \textbf{1} - a \not \in Inv (\mathcal{A}) \}$ and the spectral radius of $a$ is $\nu(a) = sup \Big\{|\lambda| \ : \ \lambda \in \mathfrak{S}(a) \Big\}.$ We know that if $|z| < 1$ ($z \in \mathbb{C}$), then we have
\begin{align*}
(1 + z)^{-n} = \sum_{k = 0}^{\infty}\Big(_{k}^{-n}\Big)z^{k}, \ for \ any \ n \in \mathbb{N},
\end{align*}
where $\Big(_{k}^{-n}\Big) = \frac{(-n)(-n - 1) ... (-n - k + 1)}{k!} = \frac{(-1)^{k}n(n + 1) ... (n + k - 1)}{k!}$. Let $z$ and $\lambda$ be two complex numbers such that $|z| < |\lambda|$. Then for any $n \in \mathbb{N}$, we see that
\begin{align}
(\lambda + z)^{-n} = \sum_{k = 0}^{\infty}\Big(_{k}^{-n}\Big)\lambda^{-n - k}z^{k}.
\end{align}

Let $\mathcal{A}$ be a unital Banach algebra and let $a$ be an arbitrary element of $\mathcal{A}$. It follows from Theorem 4.2.2 of \cite{D1} that if $f = \sum_{k = 0}^{\infty}\alpha_k Z^{k}$ has radius of convergence $r$, where $r > \nu(a)$, then $f(a) = \sum_{k = 0}^{\infty}\alpha_k a^{k}$. Here $Z$ is the coordinate functional, so that $Z : z \rightarrow z$ on $\mathbb{C}$. Let $\lambda$ be a complex number such that $\nu(a) < |\lambda|$ and let $f = (\lambda + Z)^{-n}$ and $\alpha_k = \Big(_{k}^{-n}\Big)\lambda^{-n - k}$. It follows from (2.6) that $f = (\lambda + Z)^{-n} = \sum_{k = 0}^{\infty}\alpha_k Z^{k}$ and therefore, we have
\begin{align}
f(a) = (\lambda \textbf{1} + a)^{-n} = \sum_{k = 0}^{\infty}\alpha_k a^{k} = \sum_{k = 0}^{\infty}\Big(_{k}^{-n}\Big)\lambda^{-n - k}a^{k}.
\end{align}
Let $a, b \in \mathcal{A}$ and $\lambda \in \mathbb{C}$ such that $\nu(a + b - \lambda \textbf{1}) < |\lambda|$. It follows from (2.7) that
$(\lambda \textbf{1} + a + b - \lambda \textbf{1})^{-n} = \sum_{k = 0}^{\infty}\Big(_{k}^{-n}\Big)\lambda^{-n - k}(a + b - \lambda \textbf{1})^{k}$. Letting $b_1 = b - \lambda \textbf{1}$, we have
\begin{align*}
(a + b)^{-n} & = \sum_{k = 0}^{\infty}\Big(_{k}^{-n}\Big)\lambda^{-n - k}(a + b_1)^{k} \\ & = \sum_{k = 0}^{\infty}\sum_{j = 0}^{k}\Big(_{k}^{-n}\Big)\Big(_{j}^{k}\Big)\lambda^{-n - k}\delta_{a + b_1, b_1}^{k - j}(\textbf{1})b_1^{j},
\end{align*}
for any $n \in \mathbb{N}$.

\vspace{.25cm}

\end{document}